\documentclass{amsart}
\RequirePackage{amsmath}
\RequirePackage{amsthm}
\RequirePackage{amsfonts}
\RequirePackage{amssymb}
\RequirePackage{multirow}
\RequirePackage{multicol}
\RequirePackage{graphicx}
\RequirePackage{float}
\RequirePackage{enumerate} 
\RequirePackage{cancel}
\RequirePackage{appendix}
\RequirePackage{calrsfs}
\RequirePackage{cite}
\RequirePackage[figurewithin=section]{caption}
\RequirePackage{xcolor}
\RequirePackage{hyperref}
\RequirePackage{tikz}
\usetikzlibrary{babel}
\usetikzlibrary{shapes,arrows}
\tikzstyle{block} = [draw, fill=blue!20, rectangle, 
    minimum height=3em, minimum width=3em]
\tikzstyle{sum} = [draw, fill=blue!20, circle, node distance=1cm]
\tikzstyle{input} = [coordinate]
\tikzstyle{output} = [coordinate]

\newtheorem{theorem}{Theorem}[section]
\newtheorem{proposition}[theorem]{Proposition}

\newtheorem{lemma}[theorem]{Lemma}

\theoremstyle{definition}

\newtheorem{definition}[theorem]{Definition}

\newtheorem{remark}[theorem]{Remark}
\newtheorem{example}[theorem]{Example}

\newcommand{\C}{\mathbb{C}}
\newcommand{\T}{\mathbb{T}}
\newcommand{\R}{\mathbb{R}}
\newcommand{\D}{\mathbb{D}}
\newcommand{\N}{\mathbb{N}}

\renewcommand{\H}{\mathbb{H}}

\newcommand{\abs}[1]{\left| #1 \right|}

\title{Examples in Discrete Iteration of Arbitrary Intervals of Slopes}

\date{\today}
\thanks{This research was supported in part by Ministerio de Innovaci\'on y Ciencia, Spain, project PID2022-136320NB-I00. The second author was supported by Ministerio de Universidades, Spain, through the action Ayuda del Programa de Formaci\'on de Profesorado Universitario, reference FPU21/00258.}
\author[M. D. Contreras]{Manuel D. Contreras}
\address{Departamento de Matem\'atica Aplicada II and IMUS, Escuela T\'ecnica Superior de Ingenier\'ia, Universidad de Sevilla,
	Camino de los Descubrimientos, s/n 41092, Sevilla, Spain}
\email{contreras@us.es}
\author[F. J. Cruz-Zamorano]{Francisco J. Cruz-Zamorano}
\address{Departamento de Matem\'atica Aplicada II and IMUS, Escuela T\'ecnica Superior de Ingenier\'ia, Universidad de Sevilla,
	Camino de los Descubrimientos, s/n 41092, Sevilla, Spain}
\email{fcruz4@us.es}
\author[L. Rodr\'iguez-Piazza]{Luis Rodr\'iguez-Piazza}
\address{Departmento de An\'alisis Matem\'atico and IMUS, Facultad de Matem\'aticas, Universidad
	de Sevilla, Calle Tarfia, s/n 41012 Sevilla, Spain}
\email{piazza@us.es}

\begin{document}
\begin{abstract}
Given a compact interval $[a,b] \subset [0,\pi]$, we construct a parabolic self-map of the upper half-plane whose set of slopes is $[a,b]$. The nature of this construction is completely discrete and explicit: we explicitly construct a self-map and we explicitly show in which way its orbits wander towards the Denjoy-Wolff point. We also analyze some properties of the Herglotz measure corresponding to such example, which yield the regularity of such self-map in its Denjoy-Wolff point.
\end{abstract}
\maketitle
\section{Introduction}
Discrete Iteration in the unit disk $\D$ is a branch of the vast field known as Complex Dynamics. Given a holomorphic self-map $g \colon \D \to \D$, the main aim is to analyze asymptotic properties of the sequence of iterated self-compositions of $g$, given by $g^0 = \mathrm{Id}_{\D}$, and $g^{n+1} = g^n \circ g$, $n \in \N \cup \{0\}$. In this article we focus on non-elliptic self-maps, which are the ones that possess no fixed point on $\D$. The following well-known result concerns the dynamics of such self-maps.
\begin{theorem}
\label{thm:DW}
{\normalfont \cite[Theorem 3.2.1]{AbateBook} (Denjoy-Wolff Theorem).}
Let $g \colon \D \to \D$ be a non-elliptic self-map. Then, there exists $\tau \in \partial\D$ such that $g^n \to \tau$ locally uniformly, as $n \to \infty$.
\end{theorem}
The point $\tau$ is usually known as the Denjoy-Wolff point of $g$. In particular, if $w_0 \in \D$, the Denjoy-Wolff Theorem proves that the orbit $w_n = g^n(w_0)$ converges to $\tau$, as $n \to \infty$.

A classical problem in Discrete Iteration is to characterize the directions through which the orbits of a non-elliptic self-map of $\D$ converge towards the Denjoy-Wolff point. Namely, given $w_0 \in \D$, calculating the numbers $s \in [-\pi/2,\pi/2]$ such that there exists a subsequence of the orbit satisfying $\arg(1-\overline{\tau}w_{n_k}) \to s$, as $k \to \infty$. This is the Slope Problem, which firstly appeared around 1930 in \cite{Valiron,WolffSlope}.

Dealing with this problem is usually easier in the upper half-plane setting. To do so, let $\H := \{z \in \C : \mathrm{Im}(z) > 0\}$, and consider the conformal mapping $S \colon \D \to \H$ given by
$$S(w) = i\dfrac{\tau + w}{\tau - w}, \qquad w \in \D,$$
where $\tau \in \partial\D$ is the Denjoy-Wolff point of $g$. Then, construct the self-map $f \colon \H \to \H$ given by $f = S \circ g \circ S^{-1}$. Similarly as in the case of the unit disk, we can consider the iterated self-compositions $f^0 = \mathrm{Id}_{\H}$, and $f^{n+1} = f^n \circ f$, $n \in \N \cup \{0\}$. It is easy to check that $f^n = S \circ g^n \circ S^{-1}$. Something similar occurs in the case of orbits. Given $z_0 \in \H$, its orbit by $f$ is the sequence given by $z_n = f^n(z_0)$. By construction $z_n \to \infty$, as $n \to \infty$. If $w_0 \in \D$ and we choose $z_0 = S(w_0)$, then $z_n = S(w_n)$. That is, $f^n \to \infty$, as $n \to \infty$, locally uniformly. So that the Denjoy-Wolff point of $f$ is infinity.

Concerning the Slope Problem, fixing a subsequence of the orbits, we have that $\arg(1-\overline{\tau}w_{n_k}) \to s$ if and only if $\arg(z_n) \to \pi/2-s$. Inspired by this relationship, we give the following definition.
\begin{definition}
Let $f \colon \H \to \H$ be a holomorphic function whose Denjoy-Wolff point is infinity. Then, given $z_0 \in \H$, we define the set of slopes of $f$ as
$$\mathrm{Slope}[f,z_0] = \{s \in [0,\pi] : \text{There exists a subsequence } \{z_{n_k}\} \text{ with } \arg(z_{n_k}) \to s\},$$
where $\arg$ denotes the principal branch of the argument function.
\end{definition}
By \cite[Proposition 2.5.8]{AbateBook}, the Denjoy-Wolff point of $f \colon \H \to \H$ is infinity if and only if
$$\alpha := \angle\lim_{z \to \infty}\dfrac{f(z)}{z} \in [1,+\infty).$$
In this situation, $f$ is said to be parabolic if $\alpha = 1$, and it is said to be hyperbolic if $\alpha > 1$.

Valiron \cite{Valiron} studied the Slope Problem for hyperbolic self-maps of $\H$ in 1931. It turns out that if $f \colon \H \to \H$ is a hyperbolic self-map with Denjoy-Wolff point infinity, then for every $z \in \H$ there exist $\theta \in (0,\pi)$ such that $\lim_{n \to \infty}\arg(f^n(z)) = \theta$; see \cite[Theorem 4.3.4]{AbateBook}. His ideas have been recently revisited by Bracci and Poggi-Corradini. More precisely, they proved that the function $\mathrm{Slope}[f,\cdot]\colon \H \to (0,\pi)$ is surjective and harmonic \cite[Property 2 (a) and Property 2 (b)]{BracciCorradiniValiron}.

The theory is a significantly more difficult in the case of parabolic self-maps, which are typically divided in two families. For a holomorphic self-map $f \colon \H \to \H$, the Schwarz-Pick Lemma assures that
$$k_{\H}(f^{n+2}(z),f^{n+1}(z)) \leq k_{\H}(f^{n+1}(z),f^n(z)), \qquad z \in \H, \, n \in \N,$$
where $k_{\H}$ denotes the hyperbolic distance of the upper half-plane. In particular, $\lim_{n \to \infty}k_{\H}(f^{n+1}(z),f^n(z))$ exists for all $z \in \H$, and it is a non-negative real number. For non-elliptic self-maps, by \cite[Corollary 4.6.9.(i)]{AbateBook}, we know that the latter limit is either positive for every $z \in \H$ or for no $z \in \H$. Accordingly, we say that $f$ is of positive or of zero hyperbolic step.

In 1979, Pommerenke contributed to the Slope Problem for parabolic self-maps of $\H$. In \cite[Remark 1]{PommerenkeHalfPlane}, he proved that $\lim_{n \to \infty}\arg(f^n(z))$ exists and it is either $0$ or $\pi$ whenever $f$ is a parabolic self-map of positive hyperbolic step with Denjoy-Wolff point infinity. In \cite[Proposition 2.6]{CCZRP-Slope}, using the hyperbolic geometry, we noticed that the latter limit does not depend on $z$. Namely, it holds that either $\mathrm{Slope}[f,z] = \{0\}$ or $\mathrm{Slope}[f,z] = \{\pi\}$ for all $z \in \H$.

The main emphasis in this paper is put on the case where $f$ is a parabolic self-map of zero hyperbolic step. One of the first surprising contributions to this case is in the remarkable paper of Wolff in 1929. In \cite[Section 6]{WolffSlope}, he came up with the map $f \colon \H \to \H$ given by
\begin{equation}
\label{eq:Wolff-example}
f(z) = z + ie^{\pi/2}z^i+ie^{\pi/2}, \quad z \in \H,
\end{equation}
where $z^i$ is defined using the principal branch of the logarithm. It is easy to check that $f$ is parabolic with Denjoy-Wolff point infinity, since $\angle\lim_{z \to \infty}f(z)/z = 1$. For this example, Wolff managed to prove that $\mathrm{Slope}[f,z_0]$ contains at least two points for all initial points $z_0 \in \H$.

The Slope Problem has also been introduced in the continuous setting of Complex Dynamics, that is, for semigroups $\{\phi_t\}_{t \geq 0}$ of holomorphic self-maps of the upper half-plane. A modern exposition of this topic can be found in \cite[Chapter 17]{BCDM}. In \cite{CDM-Slope-Pacific}, Contreras and D\'iaz-Madrigal found that the results of Valiron and Pommerenke can be translated to hyperbolic and parabolic semigroups of positive hyperbolic step, respectively. For a parabolic semigroup of zero hyperbolic step, they proved that the set of slopes is a compact interval which does not depend on the initial point. Despite Wolff's example \eqref{eq:Wolff-example}, they conjectured that, in the continuous setting, the set of slopes should always be a singleton. This conjecture was disproved by Contreras, D\'iaz-Madrigal, and Gumenyuk \cite{CDMGSlope}, and also by Betsakos \cite{BetsakosSlope}, indepedently. Namely, they found a semigroup whose set of slopes is the full interval $[0,\pi]$. Improving the techniques by Betsakos, for every prefixed interval $[a,b] \subset [0,\pi]$, Kelgiannis \cite{KelgiannisSlope} constructed a semigroup whose set of slopes is $[a,b]$.

Quite recently, the latter ideas were taken back to the discrete setting. Namely, even if the orbits of a self-map are discrete sets, we have shown that the set of slopes of parabolic self-maps of zero hyperbolic step are always a compact interval which does not depend on the initial point.
\begin{theorem}
\label{thm:slope-0hs} 
{\normalfont \cite[Theorem 2.9]{CCZRP-Slope}}
Let $f \colon \H \to \H$ be a parabolic function of zero hyperbolic step whose Denjoy-Wolff point is infinity. Then, there exists $0 \leq a \leq b \leq \pi$ such that $\mathrm{Slope}[f,z] = [a,b]$ for all $z \in \H$.
\end{theorem}

Examples with arbitrary set of slopes have also been constructed through the use of semigroups. Namely, if $f = \phi_1$, where $\{\phi_t\}$ is the semigroup constructed by Kelgiannis, we have proved that $\mathrm{Slope}[f,z] = [a,b]$ for all $z \in \H$. Therefore, we have the following result.
\begin{theorem}
{\normalfont \cite[Theorem 2.13]{CCZRP-Slope}}
\label{thm:arbitrary-slope}
Given $0 \leq a \leq b \leq \pi$, there exists a parabolic function $f \colon \H \to \H$ of zero hyperbolic step whose Denjoy-Wolff point is infinity such that $\mathrm{Slope}[f,z] = [a,b]$ for all $z \in \H$.
\end{theorem}

This note focuses on constructing examples of parabolic self-maps of zero hyperbolic step with arbitrary slope sets $[a,b] \subset [0,\pi]$, with $[a,b] \neq [0,\pi]$. These self-maps are of a discrete nature, and so we extend a previous work that primarily relied on semigroups. This yields a new proof of Theorem \ref{thm:arbitrary-slope} using tools from Discrete Iteration, which is given in Section \ref{sec:main}. Indeed, our construction is completely explicit: we explicitly construct the self-map and we explicitly show in which way the orbits wander towards the Denjoy-Wolff point. This strongly contrasts with the arguments in \cite{KelgiannisSlope}, where the semigroup is obtained through its Koenigs domain and the proof relies in subtle estimates of some harmonic measures. In Section \ref{sec:regularity}, we analyze some properties of the Herglotz measure corresponding to such examples. This leads us to discuss the regularity of such self-maps at their Denjoy-Wolff points, and we compare this fact with previous references.

In \cite[Section 4]{CCZRP-Slope}, we also gave a discrete and explicit construction for the cases $[0,\pi]$ and $[0,\pi/2]$. Those self-maps are defined through their Herglotz representation formula, and we discuss their regularity at the Denjoy-Wolff point. However, the new construction is easier to follow.

\section{The main construction}
\label{sec:main}
The goal of this section is, given $0 \leq a < b \leq \pi$ with $[a,b] \neq [0,\pi]$, to construct a parabolic self-map $f \colon \H \to \H$ of zero hyperbolic step with Denjoy-Wolff point infinity for which
$$\mathrm{Slope}[f,z] = [a,b], \qquad z \in \H.$$
Notice that the limit case $[a,b] = [0,\pi]$ has been covered in \cite[Theorem 4.2]{CCZRP-Slope} with a different technique.

We will start with a general setting, which we appropriately modify later to get the desired examples. To present it, let us fix some notation. Consider $\Omega = \C \setminus (-\infty,0]$, $\theta \in (0,\pi/2)$, and $A_{\theta} = \{z \in \C : \arg(z) \in (-\theta,\theta)\}$. We will construct a function $F \colon \Omega \to \C$ such that $F(A_{\theta}) \subset A_{\theta}$. The function $F$ is given by
\begin{equation}
\label{eq:map-F}
F(z) = z + p(z), \qquad z \in \Omega,
\end{equation}
where $p \colon \Omega \to \C$ is given by
\begin{equation}
\label{eq:map-p}
p(z) = \sum_{k = 1}^{\infty}p_k(z), \qquad p_k(z) = \dfrac{a_ke^{i\theta_k}}{(z+\gamma_k)^{\epsilon_k}}, \qquad z \in \Omega,
\end{equation}
where we are using the main branch of the argument of $w$ to define $w^{\epsilon_k}$. In the definition of $p$, the coefficients $a_k$, $\gamma_k$, $\epsilon_k$ and $\theta_k$ satisfy the following assumptions:
\begin{align}
a_k > 0, \qquad \gamma_k > 0, \qquad \epsilon_k > 0, \qquad \pi\epsilon_k \leq \theta, \qquad \lim_{k \to \infty}\epsilon_k = 0, \label{cond:all-positive} \\
\theta_{2k}+\pi\epsilon_{2k} = \theta, \qquad \theta_{2k-1}-\pi\epsilon_{2k-1} = -\theta, \label{cond:angles} \\
\sum_{l = 1}^{\infty}\dfrac{a_l}{\gamma_l^{\epsilon_l}} \leq \dfrac{\gamma_1}{2}, \label{cond:step} \\
\gamma_1 \geq 2, \qquad \gamma_{k+1} > \gamma_k^2, \label{cond:big-gamma} \\
\sum_{l = 1}^{k-1}\dfrac{a_l}{\gamma_k^{\epsilon_l}} \leq \dfrac{\Delta(\theta)a_k}{2k\gamma_k^{2\epsilon_k}}, \qquad \sum_{l = k+1}^{\infty}\dfrac{a_l}{\gamma_l^{\epsilon_l}} \leq \dfrac{\Delta(\theta)a_k}{2k\gamma_k^{2\epsilon_k}}, \label{cond:small-arg}
\end{align}
where
$$\Delta(\theta) = \dfrac{1}{\sqrt{4+\tan^2(\theta)}}.$$

\begin{example}
\label{ex:coefficients-arbitrary-slope}
The latter conditions are fulfilled for the coefficients
$$\epsilon_k = \dfrac{\theta}{\pi}\dfrac{1}{2^k}, \qquad a_k = C_1^k(k!)^2, \qquad \gamma_k^{\epsilon_k} = (C_2 k!)^{3^k}, \qquad k \in \N,$$
where $C_1,C_2 > 1$ are large enough constants, and $\theta_k$ is automatically defined by \eqref{cond:angles}.

\end{example}

\begin{lemma}
\label{lemma:invariant-angle}
Assume \eqref{cond:all-positive}-\eqref{cond:big-gamma}. Then, $p$ is a well-defined holomorphic function on $\Omega$. Moreover, the image of $p$ is contained in $A_{\theta}$.
\begin{proof}
Notice that $p$ is a sum of holomorphic maps which are well-defined on $\Omega$. Then, to see that $p$ is well-defined and holomorphic on $\Omega$, it is enough to find that the sum defining $p$ is uniformly convergent on every compact subset of $\Omega$. Indeed, if $K \subset \Omega$ is a compact set, then there must exist $k_0 \in \N$ such that $K \subset \{z \in \C : \abs{z} \leq \gamma_{k_0}/2\}$. Then, there exists $M > 0$ such that
$$\sum_{l = 1}^{k_0}\abs{p_l(z)} \leq M, \qquad z \in K.$$
For the other terms, if $l > k_0$, by \eqref{cond:big-gamma} we have 
$$\gamma_l \geq \gamma_{l-1}^2 \geq 2\gamma_{l-1} \geq 2\gamma_{k_0}.$$
Then, if $l > k_0$ and $z \in K$, we get
$$\abs{z+\gamma_l} \geq \gamma_l-\dfrac{\gamma_{k_0}}{2} = \gamma_l\left(1-\dfrac{\gamma_{k_0}}{2\gamma_l}\right) \geq \dfrac{\gamma_l}{2}.$$
Then, using the latter inequality and \eqref{cond:all-positive}, for $l > k_0$ we have that
$$\abs{p_l(z)} =\dfrac{a_l}{\abs{z+\gamma_l}^{\epsilon_l}} \leq \dfrac{2^{\epsilon_l}a_l}{\gamma_l^{\epsilon_l}} \leq \dfrac{2a_l}{\gamma_l^{\epsilon_l}} .$$
All in all, using also \eqref{cond:step} and Weierstrass Theorem, we conclude that the sum defining $p$ is uniformly convergent on $K$.

Let us now prove that the image of $p$ is contained in $A_{\theta}$. Since $A_{\theta}$ is a convex cone, for all $z,w \in A_{\theta}$ we have that $z+w \in A_{\theta}$. Moreover, if $z \in A_{\theta}$ and $w \in \overline{A_{\theta}}$, it is easy to check that $z+w \in A_{\theta}$. Summing up, it is enough to prove that the image of each $p_k$ is contained in $A_{\theta}$. To do so, let us assume that $k \in \N$ is even (if it is odd, then the argument is similar). Consider the map
$$z \mapsto \dfrac{1}{(z+\gamma_k)^{\epsilon_k}}, \qquad z \in \Omega.$$
By \eqref{cond:all-positive}, its image is the angular region $A = \{z \in \C : \arg{z} \in (-\pi\epsilon_k,\pi\epsilon_k)\}$. Then, by \eqref{cond:all-positive} and \eqref{cond:angles}, the image of $p_k$ is the angular region delimited by the arguments
$$\theta_k+\pi\epsilon_k = \theta, \qquad \theta_k-\pi\epsilon_k = \theta-2\pi\epsilon_k \geq \theta-2\theta \geq -\theta.$$
Therefore, it is clear that the image of $p_k$ is contained in $A_{\theta}$.
\end{proof}
\end{lemma}

\begin{remark}
\label{remark:p-is-holomorphic}
Let us notice that the first part of the latter proof can easily be adapted to see that $p$ is a well-defined and holomorphic map on $\C \setminus (-\infty,-\gamma_1]$.
\end{remark}

\begin{remark}
\label{remark:F-self-map}
The latter lemma implies that the map $F$ given in \eqref{eq:map-F} satisfies $F(A_{\theta}) \subset A_{\theta}$, since $A_{\theta}$ is a convex cone. Moreover, the same argument yields that $F(H) \subset H$ for every half-plane $H \subset \Omega$ such that $A_{\theta} \subset H$.
\end{remark}

The next result is the main idea behind the examples that we will construct later.
\begin{lemma}
\label{lemma:main-idea-slope-arbitrary}
Assume \eqref{cond:all-positive}-\eqref{cond:small-arg}. Let $z_0 = \gamma_1/2$, and consider $z_n = x_n+iy_n = F^n(z_0)$, $n \in \N$. There exist two subsequences $z_{n_k}$ and $z_{m_k}$ such that $\arg(z_{n_k}) \to \theta$ and $\arg(z_{m_k}) \to -\theta$ as $k \to \infty$.
\begin{proof}
First of all, since $z_0 \in A_{\theta}$, it follows from Remark \ref{remark:F-self-map} that $z_n \in A_{\theta}$ for all $n \in \N$. Moreover, since $\theta \in (0,\pi/2)$, we have that $\mathrm{Re}(p(z)) > 0$ for all $z \in \Omega$. This means that $x_n$ is an increasing sequence and that $F$ has no fixed point on $\overline{A_{\theta}}$. Therefore, using Remark \ref{remark:F-self-map}, we conclude that the Denjoy-Wolff point of $F \colon A_{\theta} \to A_{\theta}$ is infinity. Then, $z_n \to \infty$, as $n \to \infty$. Since $\abs{z}$ and $\mathrm{Re}(z)$ are comparable quantities for all $z \in A_{\theta}$, we deduce that $x_n \to +\infty$, as $n \to \infty$.

Following this idea, we define
$$\Omega_k := \{z \in A_{\theta} : \gamma_k \leq \mathrm{Re}(z) \leq \gamma_k^2\}, \quad k \in \N.$$
By \eqref{cond:big-gamma}, $\Omega_k \cap \Omega_l = \emptyset$ for all $k \neq l$. We claim that, for every $k \in \N$ there exists $N_k \in \N$ such that $z_{N_k} \in \Omega_k$ and $z_n \not\in \Omega_k$ for all $n < N_k$. To see this, use \eqref{cond:step} and notice that
\begin{equation}
\label{eq:bound-for-the-x-step}
x_{n+1}-x_n = \mathrm{Re}(p(z_n)) \leq \abs{p(z_n)} \leq \sum_{k = 1}^{\infty}\dfrac{a_k}{\abs{z_n+\gamma_k}^{\epsilon_k}} \leq \sum_{k = 1}^{\infty}\dfrac{a_k}{\gamma_k^{\epsilon_k}} \leq \dfrac{\gamma_1}{2},
\end{equation}
where we have used that $\mathrm{Re}(z_n) \geq 0$. Since
$$\gamma_k^2-\gamma_k \geq \gamma_k \geq \gamma_1 > \dfrac{\gamma_1}{2},$$
where we have also used \eqref{cond:big-gamma}, the claim follows.

Recall that $x_n \to +\infty$, as $n \to \infty$. Since $\{z_n\} \subset A_{\theta}$, we can find $\widetilde{N}_k \in \N$ such that $z_n \in \Omega_k$ for all $N_k \leq n < \widetilde{N}_k$ and $z_{\widetilde{N}_k} \not\in \Omega_k$.

The constructions of the subsequences $z_{n_k}$ and $z_{m_k}$ are very similar, so we will only construct $z_{n_k}$. To do so, assume that $k \in \N$ is even. Notice that, by \eqref{cond:angles},
$$\arg(p_k(z_n)) \in (\theta-2\pi\epsilon_k,\theta), \quad N_k \leq n < \widetilde{N}_k.$$
Moreover, if $l < k$ and $N_k \leq n < \widetilde{N}_k$, then
$$\abs{p_l(z_n)} \leq \dfrac{a_l}{\gamma_k^{\epsilon_l}}.$$
Similarly, if $l > k$ and $N_k \leq n < \widetilde{N}_k$, then
$$\abs{p_l(z_n)} \leq \dfrac{a_l}{\gamma_l^{\epsilon_l}}.$$
However, if $N_k \leq n < \widetilde{N}_k$, we use that $\Omega_k$ is a trapezium whose furthest points from $-\gamma_k$ are $\gamma_k^2 \pm i\gamma_k^2\tan(\theta)$ to conclude that
\begin{align}
\abs{p_k(z_n)} & \geq \dfrac{a_k}{\abs{\gamma_k^2+\gamma_k+i\gamma_k^2\tan(\theta)}^{\epsilon_k}} \geq \dfrac{a_k}{\abs{2\gamma_k^2+i\gamma_k^2\tan(\theta)}^{\epsilon_k}} \nonumber \\
& = \dfrac{a_k}{\gamma_k^{2\epsilon_k}\abs{2+i\tan(\theta)}^{\epsilon_k}} = \dfrac{\Delta(\theta)^{\epsilon_k}a_k}{\gamma_k^{2\epsilon_k}} \geq \dfrac{\Delta(\theta)a_k}{\gamma_k^{2\epsilon_k}}, \label{eq:lower-estimate-pk}
\end{align}
where we have used that $\Delta(\theta) \in (0,1)$ and that $\epsilon_k \leq \theta/\pi < 1$, by \eqref{cond:all-positive}. 

All in all, let us define
$$R_k(z_n) := \dfrac{\displaystyle\sum_{l \in \N, \, l \neq k}p_l(z_n)}{p_k(z_n)}, \qquad N_k \leq n < \widetilde{N}_k.$$
Using \eqref{cond:small-arg} and \eqref{eq:lower-estimate-pk}, it is clear that
$$\abs{R_k(z_n)} \leq \dfrac{\displaystyle\sum_{l \in \N, \, l \neq k}\abs{p_l(z_n)}}{\abs{p_k(z_n)}} \leq \dfrac{1}{k}, \qquad N_k \leq n < \widetilde{N}_k.$$
In particular,
$$\arg(p(z_n)) = \arg(p_k(z_n)) + \arg\left(1+R_k(z_n)\right),$$
where
$$\abs{\arg\left(1+R_k(z_n)\right)} \leq \arcsin\left(\dfrac{1}{k}\right) \leq \dfrac{2}{k}.$$
We conclude that
$$\arg(p(z_n)) \geq \theta - 2\pi\epsilon_k - 2/k, \qquad N_k \leq n < \widetilde{N}_k.$$
Therefore,
\begin{equation}
\label{eq:estimate-quotient}
\dfrac{y_{n+1}-y_n}{x_{n+1}-x_n} \geq \tan\left(\theta - 2\pi\epsilon_k - 2/k\right), \qquad N_k \leq n < \widetilde{N}_k.
\end{equation}

By \eqref{eq:bound-for-the-x-step}, it is clear that
$$x_{N_k} \leq \gamma_k + \dfrac{\gamma_1}{2}.$$
Since $z_{N_k} \in A_{\theta}$, this means that
$$y_{N_k} \geq -\left(\gamma_k + \dfrac{\gamma_1}{2}\right)\tan(\theta).$$
All in all, using \eqref{eq:estimate-quotient}, for $k \in \N$ large enough so that $\tan\left(\theta - 2\pi\epsilon_k - 2/k\right) \geq 0$, we have that
\begin{align*}
y_{\widetilde{N}_k} & = y_{N_k} + (y_{\widetilde{N}_k}-y_{N_k}) \geq y_{N_k} + (x_{\widetilde{N}_k}-x_{N_k})\tan(\theta - 2\pi\epsilon_k - 2/k) \\
& \geq -\left(\gamma_k + \dfrac{\gamma_1}{2}\right)\tan(\theta) + \left(\gamma_k^2-\gamma_k-\dfrac{\gamma_1}{2}\right)\tan\left(\theta - 2\pi\epsilon_k - \dfrac{2}{k}\right).
\end{align*}
Then,
$$\dfrac{y_{\widetilde{N}_k}}{x_{\widetilde{N}_k}} \geq \dfrac{-\left(\gamma_k + \dfrac{\gamma_1}{2}\right)\tan(\theta) + \left(\gamma_k^2-\gamma_k-\dfrac{\gamma_1}{2}\right)\tan\left(\theta - 2\pi\epsilon_k - \dfrac{2}{k}\right)}{\gamma_k^2+\dfrac{\gamma_1}{2}}.$$
In particular, since we chose $k \in \N$ to be an even number, we use \eqref{cond:all-positive} and \eqref{cond:big-gamma} to conclude that
$$\theta \geq \liminf_{k \to \infty}\arg(z_{\widetilde{N}_{2k}}) \geq \theta.$$
Namely,
$$\lim_{k \to \infty}\arg(z_{\widetilde{N}_{2k}}) = \theta,$$
and so it is clear that it is enough to choose $n_k = \widetilde{N}_{2k}$.
\end{proof}
\end{lemma}

Let us now address the main construction.
\begin{theorem}
\label{thm:arbitrary-slope-examples}
Let $0 \leq a < b \leq \pi$, $[a,b] \neq [0,\pi]$. There exists a parabolic self-map $f \colon \H \to \H$ of zero hyperbolic step with Denjoy-Wolff point infinity for which
$$\mathrm{Slope}[f,z] = [a,b], \qquad z \in \H.$$
\begin{proof}
Consider the angle
$$\theta = \dfrac{b-a}{2} \in (0,\pi/2),$$
and the map $F$ as in \eqref{eq:map-F}, where we are assuming \eqref{cond:all-positive}-\eqref{cond:small-arg}. Let us also consider the half-plane
$H = \{z \in \C : \arg(z) \in (-\theta-a,\pi-\theta-a)\}$. Notice that $A_{\theta} \subset H$. Then, as seen in Remark \ref{remark:F-self-map}, $F(H) \subset H$. In that case, consider the restriction $g = F_{|H} \colon H \to H$.

Let us now remark that the map $p \colon H \to A_{\theta}$, as defined in \eqref{eq:map-p}, is bounded. To see this, notice that
$$\min_{z \in \partial H}\abs{z+\gamma_k} = \gamma_k\sin(a+\theta).$$
Then, by \eqref{cond:step}, given that $z \in H$, we have
\begin{equation}
\label{eq:p-is-bounded}
\abs{p(z)} \leq \sum_{k = 1}^{\infty}\abs{p_k(z)} \leq \dfrac{1}{\sin(a+\theta)}\sum_{k = 1}^{\infty}\dfrac{a_k}{\gamma_k^{\epsilon_k}} < +\infty,
\end{equation}
where we have used that $\sin(a+\theta) \in (0,1]$ and $\epsilon_k \leq 1$, by \eqref{cond:all-positive}. Then, it is clear that
\begin{equation}
\label{eq:g-is-parabolic}
\lim_{\substack{z \to \infty \\ z \in H}}\dfrac{g(z)}{z} = 1.
\end{equation}

Now, consider the conjugation $f \colon \H \to \H$ given by 
\begin{equation}
\label{eq:def-conjutaged-f}
f(z) = \xi g(\overline{\xi} z), \qquad z \in \H ,
\end{equation}
where $\xi = \exp(i(a+\theta))$. By \eqref{eq:g-is-parabolic}, $f$ is also parabolic with Denjoy-Wolff point infinity. Moreover, it follows from Lemma \ref{lemma:invariant-angle} that
$$\mathrm{Slope}[f,z] \subset [a,b], \qquad z \in \H.$$
In particular, taking $z_0 = \xi\gamma_1/2 \in \H$, Lemma \ref{lemma:main-idea-slope-arbitrary} constructs two subsequences of $z_n = f^n(z_0)$ such that 
$$\lim_{k \to \infty}\arg(z_{n_k}) = b, \qquad \lim_{k \to \infty}\arg(z_{m_k}) = a.$$
Together with Theorem \ref{thm:slope-0hs}, this means that
$$\mathrm{Slope}[f,z] = [a,b], \qquad z \in \H.$$
Since $a < b$, we get that $f$ is of zero hyperbolic step (see  \cite[Proposition 2.6]{CCZRP-Slope} or \cite[Remark 1]{PommerenkeHalfPlane}).
\end{proof}
\end{theorem}

\section{Regularity of the examples}
\label{sec:regularity}
We can also prove that the latter examples enjoy some regularity at the Denjoy-Wolff point if further assumptions are made. To do so, we will use the following representation result due to Herglotz.

\begin{theorem}
\label{thm:Herglotz-general}
{\normalfont\cite[Theorem 6.2.1]{Aaronson}}
Every holomorphic function $f \colon \H \to \H$ can uniquely be written as
\begin{equation}
f(z) = \alpha z + \beta + \int_{\R}\dfrac{1+tz}{t-z}d\mu(t), \quad z \in \H,
\end{equation}
where $\alpha \geq 0$, $\beta \in \R$, and $\mu$ is a positive finite measure on $\R$.
\end{theorem}

The angular derivative of $f$ at infinity can be characterized in terms of the latter representation. Namely,
$$\angle\lim_{z \to \infty}\dfrac{f(z)}{z} = \alpha,$$
as discussed in \cite[Chapter 5, Lemma 2]{DM}. In particular, we see that infinity is the Denjoy-Wolff point of $f$ if and only if $\alpha \geq 1$. Moreover, in this case, $f$ is parabolic if $\alpha = 1$ and hyperbolic if $\alpha > 1$.

Remember that, up to a conjugation, one can always suppose that a parabolic map has Denjoy-Wolff point infinity. Therefore, in this sense, the following result can be used to describe the behaviour of all parabolic functions.
\begin{theorem}
\label{thm:Herglotz}
{\normalfont \cite[Theorem 1.3]{HS}\cite[Theorem 1.2]{CZ-FiniteShift}}
A holomorphic self-map $f \colon \H \to \H$ is parabolic with Denjoy-Wolff point infinity if and only if it can be written as
\begin{equation}
\label{eq:Herglotz}
f(z) = z + \beta + \int_{\R}\dfrac{1+tz}{t-z}d\mu(t), \quad z \in \H,
\end{equation}
where $\beta \in \R$ and $\mu$ is a positive finite measure on $\R$, both of them not simultaneously null.
\end{theorem}
Using this result, for a parabolic function $f \colon \H \to \H$ with Denjoy-Wolff point infinity, we say that $f$ is {\sl represented by the pair} $(\beta,\mu)$ whenever \eqref{eq:Herglotz} holds. This notation is used in the following proposition.
\begin{proposition}
\label{prop:regularity-example-slope}
Assume \eqref{cond:all-positive}-\eqref{cond:small-arg}. Let $f \colon \H \to \H$ be the parabolic self-map constructed in the proof of Theorem \ref{thm:arbitrary-slope-examples}. Suppose that $f$ is represented by the pair $(\beta,\mu)$. Then, the measure $\mu$ is absolutely continuous with respect to Lebesgue measure on $\R$.

Moreover, if 
\begin{equation}
\label{cond:sum-for-l1}
\sum_{k = 1}^{\infty}\dfrac{a_k\log(1+\gamma_k^2)}{\gamma_k^{\epsilon_k}} < +\infty, \qquad \sum_{k = 1}^{\infty}\dfrac{a_k}{\epsilon_k\gamma_k^{\epsilon_k}} < +\infty,
\end{equation}
then
$$\int_{\R}\abs{t}d\mu(t) < +\infty.$$
In that case,
$$\beta = \int_{\R}td\mu(t).$$
\end{proposition}

A proof of Proposition \ref{prop:regularity-example-slope} will be given at the end of the section. Before doing so, we should discuss the way in which this results shows that the map $f$ constructed in Theorem \ref{thm:arbitrary-slope-examples} enjoys some regularity at the Denjoy-Wolff point.

In \cite{CDP-C2}, Contreras, D\'iaz-Madrigal, and Pommerenke studied several dynamical aspects of the iterates of a self-map with additional regularity at its Denjoy-Wolff point. To introduce their ideas, consider a holomorphic map $\phi \colon \D \to \C$ with a boundary fixed point $\tau \in \T$, that is,
$$\angle\lim_{z \to \tau}\phi(z) = \tau.$$
The map $\phi$ is said to be of angular-class of order $p \in \N$ at $\tau$, denoted as $\phi \in C_A^p(\tau)$, if there exist $c_1,\ldots,c_p \in \C$ and a holomorphic function $\gamma \colon \D \to \C$ such that
$$\phi(z) = \tau + \sum_{k = 1}^p\dfrac{c_k}{k!}(z-\tau)^k+\gamma(z), \quad z \in \D, \quad \angle\lim_{z \to \tau}\dfrac{\gamma(z)}{(z-\tau)^p} = 0.$$

In particular, as a consequence of Julia-Wolff-Carath\'eodory Theorem \cite[Corollary 2.5.5]{AbateBook}, every non-elliptic function is of angular-class of first order at its Denjoy-Wolff point.

There are some necessary conditions in terms of the Herglotz's representation \eqref{eq:Herglotz} to assure that a map enjoys a certain angular regularity. We use here the following result, which follows from \cite[Proposition 2.1]{CDP-C2} and \cite[Lemma 3.3]{HS}. It has already been noticed in \cite[Proposition 3.6]{CZ-FiniteShift} and \cite[Remark 4.6]{CCZRP-Slope}.
\begin{proposition}
Let $f \colon \H \to \H$ be a parabolic self-map whose Denjoy-Wolff point is infinity, represented by the pair $(\beta,\mu)$. Suppose that $\int_{\R}\abs{t}d\mu(t) < +\infty$. Consider the self-map $\phi \colon \D \to \D$ which is conjugated to $f$ and has Denjoy-Wolff point $\tau \in \partial\D$. Then, $\phi \in C_A^2(\tau)$. 
\end{proposition}
Therefore, if we further assume \eqref{cond:sum-for-l1}, the self-map constructed in the proof of Theorem \ref{thm:arbitrary-slope-examples} enjoys the regularity presented in the latter proposition.

It is worth mentioning that there are several previous references in the literature concerning the Slope Problem and the aforementioned regularity. More precisely, such references prove that the orbits of a semigroup have to converge with a definite slope if its Koenigs function or its infinitesimal generator enjoy some regularity. For instance, see \cite[Proposition 3.1]{Regularity-1} and \cite[Theorem 12]{Regularity-2}. Our constructions are in contrast with such results, since the  aforementioned regularity of the self-maps can be related with the regularity of their Koenigs maps, as in \cite[Theorem 4.1]{CDP-C2}, or its derivative, as in \cite[Theorems 5.2 and 6.2]{CDP-C2}. For semigroups, the derivative of the Koenigs map is related to the so-called infinitesimal generator, see \cite[Chapter 10]{BCDM}.

This is also an improvement in comparison to the examples given in \cite{CDMGSlope,BetsakosSlope,KelgiannisSlope}, where the regularity of the associated semigroup is unclear. Let us also notice that the example by Wolff given in \eqref{eq:Wolff-example} is not regular in the above sense.

\begin{proof}[Proof of Proposition \ref{prop:regularity-example-slope}]
In order to prove that $\mu$ is absolutely continuous with respect to Lebesgue measure on $\R$, we define $v = \mathrm{Im}(f(z)-z) \colon \H \to (0,+\infty)$. Notice that $v$ is a harmonic map on $\H$. By \eqref{eq:p-is-bounded}, $v$ is bounded. Therefore, it can be represented as the Poisson integral of its boundary values \cite[Chapter I, Lemma 3.4]{Garnett-Bounded}, that is,
$$v(z) = \int_{\R}\dfrac{1}{\pi}\dfrac{y}{(t-x)^2+y^2}v^*(t)dt, \qquad z = x+iy \in \H,$$
where
$$v^*(x) = \lim_{y \to 0^+}v(x+iy) \in [0,+\infty).$$
By Remark \ref{remark:p-is-holomorphic}, $f$ is holomorphic in a domain containing $\overline{\H}$. Then, the latter limit exists for all $x \in \R$.

Recall that
$$\mathrm{Im}(f(z)-z) = v(z) = \int_{\R}\dfrac{(1+t^2)y}{(t-x)^2+y^2}d\mu(t), \qquad z \in \H.$$
Then, using the uniqueness in Theorem \ref{thm:Herglotz-general}, we have that
$$\dfrac{d\mu}{dm}(t) = \dfrac{1}{\pi}\dfrac{v^*(t)}{1+t^2}, \quad t \in \R,$$
where $m$ is Lebesgue measure on $\R$.

Let us now prove that, if \eqref{cond:sum-for-l1} holds, then
$$\int_{\R}\abs{t}d\mu(t) < +\infty.$$
To do so, using \eqref{eq:def-conjutaged-f}, we note that
\begin{align}
\label{eq:int-to-sums}
\int_{\R}\abs{t}d\mu(t) & = \int_{\R}\dfrac{\abs{t}}{1+t^2}v^*(t)dt \leq \int_{\R}\dfrac{\abs{t}}{1+t^2}\abs{p(\overline{\xi}t)}dt \\
& \leq \sum_{k=1}^{\infty}\int_{\R}\dfrac{a_k\abs{t}}{(1+t^2)\abs{\overline{\xi}t+\gamma_k}^{\epsilon_k}}dt. \nonumber
\end{align}
Let us estimate the latter integrals. To start, notice that
\begin{align*}
\int_{\R}\dfrac{\abs{t}}{(1+t^2)\abs{\overline{\xi}t+\gamma_k}^{\epsilon_k}}dt = \int_{\R}\dfrac{\abs{t}}{(1+t^2)\abs{t+\cos(\phi)\gamma_k+i\sin(\phi)\gamma_k}^{\epsilon_k}}dt,
\end{align*}
where we recall that $\xi = e^{i\phi}$, $\phi = a + \theta$, and $\theta = (b-a)/2$. Then, since $0 \leq a < b \leq \pi$, we conclude that $\phi \in (0,\pi)$. For the estimates, we assume that $\phi \in (0,\pi/2]$ (the other case is similar). Indeed, we split the latter integral into three parts. For the first one, we have
\begin{align*}
\int_{-\gamma_k}^{\gamma_k}\dfrac{\abs{t}}{(1+t^2)\abs{t+\cos(\phi)\gamma_k+i\sin(\phi)\gamma_k}^{\epsilon_k}}dt & \leq \int_{-\gamma_k}^{\gamma_k}\dfrac{\abs{t}}{(1+t^2)\sin(\phi)^{\epsilon_k}\gamma_k^{\epsilon_k}}dt \\
& = \dfrac{\log(1+\gamma_k^2)}{\sin(\phi)^{\epsilon_k}\gamma_k^{\epsilon_k}} \leq \dfrac{\log(1+\gamma_k^2)}{\sin(\phi)\gamma_k^{\epsilon_k}},
\end{align*}
where we have used that $\sin(\phi) \in (0,1]$ and $\epsilon_k \in (0,1)$. Concerning the second part,
\begin{align*}
\int_{\gamma_k}^{+\infty}\dfrac{\abs{t}}{(1+t^2)\abs{t+\cos(\phi)\gamma_k+i\sin(\phi)\gamma_k}^{\epsilon_k}}dt & \leq \int_{\gamma_k}^{+\infty}\dfrac{t}{(1+t^2)t^{\epsilon_k}}dt \\
& \leq \int_{\gamma_k}^{+\infty}\dfrac{dt}{t^{1+\epsilon_k}} = \dfrac{1}{\epsilon_k\gamma_k^{\epsilon_k}}.
\end{align*}
For the third part, we have that
\begin{align*}
& \int_{-\infty}^{-\gamma_k}\dfrac{\abs{t}}{(1+t^2)\abs{t+\cos(\phi)\gamma_k+i\sin(\phi)\gamma_k}^{\epsilon_k}}dt \\
& = \int_{\gamma_k}^{+\infty}\dfrac{t}{(1+t^2)\abs{-t+\cos(\phi)\gamma_k+i\sin(\phi)\gamma_k}^{\epsilon_k}}dt \\
& \leq \int_{\gamma_k}^{+\infty}\dfrac{t}{(1+t^2)(t-\cos(\phi)\gamma_k)^{\epsilon_k}}dt.
\end{align*}
However, note that
$$t \mapsto \dfrac{t}{t-\cos(\phi)\gamma_k}, \qquad t \geq \gamma_k,$$
is a decreasing map. Therefore, we have that
$$\dfrac{t}{t-\cos(\phi)\gamma_k} \leq \dfrac{1}{1-\cos(\phi)}, \qquad t \geq \gamma_k.$$
This means that
\begin{align*}
\int_{-\infty}^{-\gamma_k}\dfrac{\abs{t}}{(1+t^2)\abs{t+\cos(\phi)\gamma_k+i\sin(\phi)\gamma_k}^{\epsilon_k}}dt & \leq \dfrac{1}{(1-\cos(\phi))^{\epsilon_k}}\int_{\gamma_k}^{+\infty}\dfrac{dt}{t^{1+\epsilon_k}} \\
& = \dfrac{1}{(1-\cos(\phi))^{\epsilon_k}}\dfrac{1}{\epsilon_k\gamma_k^{\epsilon_k}} \\
& \leq \dfrac{1}{1-\cos(\phi)}\dfrac{1}{\epsilon_k\gamma_k^{\epsilon_k}},
\end{align*}
where we have used that $1-\cos(\phi) \in (0,1]$ and $\epsilon_k \in (0,1)$.

All in all, we have shown that there exists $C = C(\phi) > 0$ so that, using \eqref{eq:int-to-sums}, we have that
$$\int_{\R}\abs{t}d\mu(t) \leq C(\phi)\sum_{k = 1}^{\infty}\left(\dfrac{a_k\log(1+\gamma_k^2)}{\gamma_k^{\epsilon_k}} + \dfrac{a_k}{\epsilon_k\gamma_k^{\epsilon_k}}\right).$$
Therefore, the result follows from \eqref{cond:sum-for-l1}.

Finally, we notice that if 
$$\int_{\R}\abs{t}d\mu(t) < +\infty,$$
then the equality
$$\beta = \int_{\R}td\mu(t)$$
is a consequence of \cite[Theorem 3.4]{CCZRP-Slope}.
\end{proof}

\bibliographystyle{plain}

\begin{thebibliography}{99}
\bibitem{Aaronson}
J. Aaronson.
\newblock {\em An Introduction to infinite ergodic theory}.
\newblock {Mathematical Surveys and Monographs} \textbf{50}. American Mathematical Society (1997). 

\bibitem{AbateBook}
M. Abate.
\newblock {\em Holomorphic Dynamics on Hyperbolic Riemann Surfaces}.
\newblock De Gruyter Studies in
Mathematics {\bf 89}. De Gruyter (2022).

\bibitem{BetsakosSlope}
D. Betsakos.
\newblock {\em On the asymptotic behavior of the trajectories of semigroups of holomorphic functions}.
\newblock {J. Geom. Anal.} {\bf 26} (2016), 557--569.

\bibitem{BCDM} 
F. Bracci, M. D. Contreras, and S. D\'iaz-Madrigal. 
\newblock {\it Continuous semigroups of holomorphic self-maps of the unit disc}.
\newblock Springer Monographs in Mathematics. Springer (2020).

\bibitem{BracciCorradiniValiron}
F. Bracci and P. Poggi-Corradini.
\newblock {\it On Valiron's theorem}.
\newblock {In Future Trends in Geometric Function Theory}, 39--55. {Rep. Univ. Jyv\"askyl\"a Dep. Math. Stat.} {\bf 92} (2003).

\bibitem{HS}
M. D. Contreras, F. J. Cruz-Zamorano, and L. Rodr\'iguez-Piazza.
\newblock {\em Characterization of the hyperbolic step of parabolic functions}.
\newblock {Comput. Methods Funct. Theory.} (2024).

\bibitem{CCZRP-Slope}
M. D. Contreras, F. J. Cruz-Zamorano, and L. Rodr\'iguez-Piazza.
\newblock {\em The Slope Problem in Discrete Iteration}.
\newblock {To appear in Discrete Contin. Dyn. Syst.} Available in ArXiv:2406.08389 (2024).

\bibitem{CDM-Slope-Pacific}
M. D. Contreras and S. D\'iaz-Madrigal.
\newblock {\em Analytic flows on the unit disk: angular derivatives and boundary fixed points.}
\newblock {Pacific J. Math.} {\bf 222} (2005), 253--286.

\bibitem{CDMGSlope}
M. D. Contreras, S. D\'iaz-Madrigal, and P. Gumenyuk.
\newblock {\it Slope problem for trajectories of holomorphic semigroups in the unit disk}.
\newblock {Comput. Methods Funct. Theory.} {\bf 15} (2014), 117--124.

\bibitem{CDP-C2}
M. D. Contreras, S. D\'iaz-Madrigal, and Ch. Pommerenke.
\newblock {\em Second angular derivatives and parabolic iteration in the unit disk}.
\newblock {Trans. Amer. Math. Soc.} {\bf 362} (2010), 357--388.

\bibitem{CZ-FiniteShift}
F. J. Cruz-Zamorano.
\newblock {\em Characterization of finite shift via Herglotz's representation.}
\newblock {J. Math. Anal. Appl.} {\bf 542} (2025), Paper No. 128883.

\bibitem{DM}
C. I. Doering and R. Ma\~n\'e.
\newblock {\em The dynamics of inner functions}. 
\newblock {Ensaios Matemáticos} {\bf 3}. Sociedade Brasileira de Matemática (1991).

\bibitem{Regularity-2}
M. Elin, D. Khavinson, S. Reich, and D. Shoikhet.
\newblock {\it Linearization models for parabolic dynamical systems via Abel’s functional equation}.
\newblock {Ann. Acad. Sci. Fenn. Math.} {\bf 35} (2010), 439--472.

\bibitem{Regularity-1}
M. Elin, S. Reich, D. Shoikhet, and F. Yacobzon.
\newblock {\it Asymptotic behavior of one-parameter semigroups and rigidity of holomorphic generators}.
\newblock {Complex Anal. Oper. Theory} {\bf 2} (2008), 55--86.

\bibitem{Garnett-Bounded}
J. B. Garnett.
\newblock {\em Bounded analytic functions.}
\newblock {Graduate Texts in Mathematics} {\bf 236}. Springer (2007).

\bibitem{KelgiannisSlope}
G. Kelgiannis.
\newblock {\it Trajectories of semigroups of holomorphic functions and harmonic measure}.
\newblock {J. Math. Anal. Appl.} {\bf 474} (2019), 1364--1374.

\bibitem{PommerenkeHalfPlane}
Ch. Pommerenke.
\newblock {\em On the iteration of analytic functions in a halfplane, {I}}.
\newblock {J. London Math. Soc. (2)} {\bf 19} (1979), 439--447.

\bibitem{Valiron}
G. Valiron.
\newblock {\em Sur l’it{\'e}ration des fonctions holomorphes dans un demi-plan}.
\newblock {Bull. Sci. Math.}  {\bf 55} (1931), 105--128.

\bibitem{WolffSlope}
J. Wolff.
\newblock {\it Sur l'it\'eration des fonctions holomorphes dans un demi-plan}.
\newblock {Bull. Soc. Math. Fr.} {\bf 57} (1929), 195--203.
\end{thebibliography}

\end{document}